\documentclass[12pt,oneside]{article}
\usepackage{amsmath,amsthm,amsfonts,amssymb,MnSymbol,mathrsfs,latexsym}
\usepackage{mathtools}
\usepackage[english]{babel}

\newcommand{\const}{\rm const}

  \newcommand{\Dom}{\rm  Dom}

\DeclareMathOperator*{\esssup}{ess\,sup}

\theoremstyle{plain}
\newtheorem{theorem}{Theorem}[section]

\newtheorem{definition}{Definition}[section]

\newcommand\myatop[2]{\genfrac{}{}{0pt}{}{#1\hfill}{#2\hfill}}

\renewenvironment{proof}{{\bf{Proof.}}}{\hfill $\Box$ \\}

\title{\large \textbf{Homogeneous kernel integral operators\\ in Grand Lebesgue Spaces}}

\footnotesize\date{}

\author{\normalsize Maria Rosaria Formica ${}^{1}$,   \normalsize Eugeny Ostrovsky
${}^2$ and \normalsize Leonid Sirota ${}^3$}

\begin{document}

  \maketitle

\begin{center}
{\footnotesize ${}^{1}$ Universit\`{a} degli Studi di Napoli \lq\lq Parthenope\rq\rq, via Generale Parisi 13,\\
Palazzo Pacanowsky, 80132,
Napoli, Italy.} \\

\vspace{1mm}

{\footnotesize e-mail: mara.formica@uniparthenope.it} \\

\vspace{2mm}

{\footnotesize ${}^{2,\, 3}$  Bar-Ilan University, Department of Mathematics and Statistics, \\
52900, Ramat Gan, Israel.} \\

\vspace{1mm}

{\footnotesize e-mail: eugostrovsky@list.ru}\\

\vspace{1mm}

{\footnotesize e-mail: sirota3@bezeqint.net} \\

\end{center}

\vspace{3mm}

\begin{abstract}
In this short report we estimate and calculate the exact value of norms of multilinear integral operators having homogeneous kernel, acting between two Grand Lebesgue Spaces.
\end{abstract}

\vspace{4mm}

{\it Key words:} \hspace{1mm} measurable spaces, Lebesgue-Riesz space, Grand Lebesgue space, multilinear integral operators, kernel, homogeneity.

 \vspace{5mm}

\section{Introduction.}

\vspace{4mm}

 \hspace{3mm} Let $ \ (X = \{x\}, \cal{F}, \mu)  \ $ and  $ \ (Y = \{y\}, \cal{M}, \nu)  \ $ be two measurable spaces with corresponding non-trivial sigma finite measures $\mu$ and $\nu$, respectively. For the numerical valued measurable functions define, as ordinarily, the classical Lebesgue-Riesz norms
\begin{equation} \label{Lp}
||g||_{p,X} = ||g||_p \stackrel{def}{=} \left[ \ \int_X |g(x)|^p \ \mu(dx) \ \right]^{1/p}, \ 1 \le p < \infty,
\end{equation}
\begin{equation} \label{Lq}
||h||_{q,Y} = ||h||_q \stackrel{def}{=} \left[ \ \int_Y |h(y)|^q \ \nu(dy) \ \right]^{1/q}, \ 1 \le q < \infty,
\end{equation}

\begin{equation*}
||g||_{\infty,X} := \esssup\limits_{x \in X} |g(x)|, \ \ (p=\infty) \ , \hspace{5mm} ||h||_{\infty,Y}:= \esssup\limits_{y \in Y} |h(y)|, \ \ (q=\infty);
\end{equation*}
and the corresponding spaces
$$
L_p(X)= \{g: X\to \mathbb R \ : \ ||g||_{p,X} < \infty\}; \ \ \ L_q(Y) = \{h: Y \to \mathbb R \ : \ ||h||_{q,Y} < \infty \}.
$$

 \vspace{4mm}

 \hspace{3mm} Let also $ \ U \ $ be an integral operator (linear or not) of the form

\begin{equation} \label{oper}
U[f](x) \stackrel{def}{=} \int_Y k(x,y, f(y)) \ \nu(dy),
\end{equation}
with the kernel (measurable) function $ \ k = k(x,y,v). \ $ \par

 \vspace{3mm}

\hspace{3mm} For these operators there are many works devoting to Lebesgue-Riesz norm estimations of the following quantity
$$
N[U]_{p,q}(s) \stackrel{def}{=} \sup_{\myatop{f \in L_{q,Y}} {||f||_{q,Y} \le s}} ||U[f]||_{p,X}, \ \  0 < s < \infty,
$$
see, for instance, \cite{liflyandostrovskysirotaturkish2010,Mitrinovich etc,Okikiolu,Stein int,Syrjanen,Wu Fu} and so on. \par

 \ Of course, for the linear operator $ \ U \ $ it is necessary to estimate the $ \ L_{q,Y}  \to L_{p,X} \ $ norm of this operator

$$
||U||_{q \to p} \stackrel{def}{=} \sup_{0 \ne f \in L_q} \left[ \ \frac{||U[f]||_p}{||f||_q} \ \right].
$$

\vspace{4mm}

\begin{definition}
{\rm The {\it multilinear} integral operator $M = M[Q](\vec{f})(x)$, where

$\vec{f}=(f_1,f_2,\ldots, f_m)$, $m = 2,3,\ldots$, and $Q$ is the (measurable) kernel,
$$
Q = Q(x; x_1,x_2,\ldots,x_m), \ \ x,x_j \in X,
$$
  has the form
\begin{equation} \label{multilin}
\begin{split}
 & \hspace{5cm}{M=M[Q](\vec{f})(x) : =}  \\[1mm]
 &  \int_{X^m} Q(x; x_1,x_2,\ldots,x_m) \ f_1(x_1) f_2(x_2) \ldots f_m(x_m) \ \mu(dx_1) \mu(dx_2) \ldots \mu(dx_m).
 \end{split}
\end{equation}
}
\end{definition}
 \ These operators have been studied in particular in the works \cite{Brenyi,Fu Grafakos Lu Zhao,Li Cen Fu Hang,Stein Harm,Wu Yan}, and so on.

 \vspace{4mm}

  \hspace{3mm} {\it We assume further that } \ $ \ X = \mathbb R_+ \ $  with ordinary Lebesgue measure $ \ d \mu = dx \ $
  {\it and that  the kernel } $ \ Q  \ $ {\it is homogeneous of the degree } $-m$, that is

\begin{equation} \label{homogen}
\forall \delta > 0 \ \Rightarrow  \ Q(\delta x; \delta x_1, \delta x_2, \ldots, \delta x_m) = \delta^{-m} \ Q(x; x_1,x_2,\ldots,x_m).
\end{equation}

 \vspace{4mm}

\ Put  $ \ \vec{p} = \{ p_1, p_2, \ldots,p_m \}$ where $p_j \in [1,\infty)$, and
\begin{equation} \label{p fun vec p}
p := \left(  \ \sum_{j=1}^m \frac{1}{p_j} \ \right)^{-1}.
\end{equation}

Let us introduce also the following auxiliary variable (function)
\begin{equation} \label{key value}
\begin{split}
 & \hspace{3cm}{\Theta_m(\vec{p}) = \Theta_m(p_1,p_2,\ldots,p_m)  \stackrel{def}{=}}  \\[2mm]
 &  \int_{\mathbb R^m_+} |Q(1,x_1,x_2, \ldots, x_m)| \ x_1^{-1/p_1} \ x_2^{-1/p_2} \ldots x_m^{-1/p_m} \ dx_1 dx_2 \ldots dx_m,
 \end{split}
\end{equation}
and define the set

\begin{equation} \label{key set}
D_m = D_m[Q] \stackrel{def}{=} \{\vec{p} = \{ p_1, p_2, \ldots,p_m \}, p_j \in [1,\infty), \  j = 1,2,\ldots,m \ : \ \Theta_m(\vec{p}) < \infty \ \}.
\end{equation}

\vspace{3mm}

 \hspace{3mm} {\it It will be presumed that this set contains some non-empty open set in} $ \mathbb R^m_+ $.

\vspace{3mm}

 \ It is proved in \cite{Brenyi} that if $\Theta_m(\vec{p}) < \infty $ for some vector $ \  \vec{p}= \{ p_1,p_2,\ldots,p_m \} $, $ \  p_j \in [1,\infty)$, then

\begin{equation} \label{important relation}
||M||_p \le \Theta_m(p_1,p_2,\ldots,p_m) \ ||f_1||_{p_1} ||f_2||_{p_2} \ldots ||f_m||_{p_m}.
\end{equation}

\vspace{3mm}

  Moreover, the last relation (\ref{important relation}) is  non-improvable for all the admissible variables $ \ m, Q \ $
 if, of course, $ \ \vec{p} \in D_m(Q) \ $ or equally  $ \ \Theta_m(p_1,p_2,\ldots,p_m) < \infty. \ $ \par

\vspace{4mm}

 \hspace{3mm} {\bf  Our aim in this short preprint is to extend  the last estimate} \eqref{important relation} {\bf from the ordinary Lebesgue-Riesz spaces to the so-called Grand Lebesgue Spaces (GLS).} \par

\vspace{4mm}

\begin{center}

 {\it A brief review about  Grand Lebesgue Spaces (GLS).}

\end{center}

\vspace{4mm}

 \hspace{3mm}   For definitions and properties of the Grand Lebesgue Spaces see, e.g.,
 \cite{Buld Koz AMS,Ermakov etc. 1986,ForKozOstr_Lithuanian,Kozachenko-Ostrovsky 1985,Kozachenko at all 2018,liflyandostrovskysirotaturkish2010,Ostrovsky1999,Ostrovsky HIAT,Ostrov Prokhorov,Ostr Integ oper,Samko-Umarkhadzhiev,Samko-Umarkhadzhiev-addendum} and so on.\par

 \vspace{3mm}

   \hspace{3mm} Let $ \ (\Omega = \{\omega\}, \cal{B}, \mu)  \ $ be a measurable space equipped with sigma finite non-trivial measure $ \ \mu $, not necessarily probabilistic. The ordinary  Lebesgue-Riesz norm for numerical valued measurable functions $ \ f: \Omega \to \mathbb R \ $ is defined by

$$
||f||_{\Omega,p} = ||f||_p =\left[ \ \int_{\Omega} |f(\omega)|^p \ \mu(d \omega) \ \right]^{1/p}, \ 1 \le p < \infty.
$$

 \ Let $ \ a = {\const} \ge 1, \ b \in (a, \infty] \ $ and $ \ \psi = \psi(p), \ p \in [a,b)$, a positive measurable numerical valued function, not necessarily finite in $b$, such that $ \ \inf_{p \in [a,b)} \psi(p) > 0$.
%

\vspace{3mm}

 For instance, if $m = {\const} > 0$,
$$
\psi_m(p) := p^{1/m},  \ \ p \in [1,\infty),
$$
or, for $ 1 \leq  a < b < \infty, \ \alpha,\beta = \const \ge 0$,
$$
   \psi_{a,b; \alpha,\beta}(p) := (p-a)^{-\alpha} \ (b-p)^{-\beta}, \  \ p \in
   (a,b),
$$
are generating functions.

The (Banach) Grand Lebesgue Space $  \ G \psi  = G\psi(a,b) $
    consists of all the real (or complex) numerical valued measurable functions
$f: \Omega \to \mathbb R$ having finite norm
\begin{equation} \label{norm psi}
    ||f||_{G\psi} \stackrel{def}{=} \sup_{p \in [a,b)} \left[ \frac{||f||_p}{\psi(p)} \right].
 \end{equation}

 \vspace{4mm}

 \ The function $ \  \psi = \psi(p) \  $ is named the {\it  generating function } for for the space $G \psi$. \par
If for instance
$$
  \psi(p) = \psi_{r}(p) = 1, \ \  p = r;  \ \  \ \psi_{r}(p) = +\infty,   \ \ p \ne r,
$$
 where $ \ r = {\const} \in [1,\infty),  \ C/\infty := 0, \ C \in \mathbb R, \ $ (an extremal case), then the corresponding $ \  G\psi^{(r)} $ space coincides  with the classical Lebesgue-Riesz space $ \ L_r = L_r(\Omega)$.

\vspace{4mm}
   The belonging of a function $ f: \Omega \to \mathbb{R}$ to some $ G\psi$
space is closely related to its tail function behavior

 $$
 T_f(t) \stackrel{def}{=} {\bf P}(|f| \ge t)={\bf \mu}(\omega:\ |f(\omega)| \ge t), \ t \ge 0,
 $$
 as $ \ t \to 0+ \ $ as well as when $ \ t \to \infty $ (see, e.g.,
\cite{Kozachenko-Ostrovsky 1985,ForKozOstr_Lithuanian}).

\hspace{3mm} In detail, let $  \upsilon  $ be a non-zero r.v. belonging to some Grand Lebesgue Space $ \ G\psi \ $ and suppose  $ \ ||\upsilon||_{G\psi} = 1 $. Define the Young-Fenchel (or Legendre) transform of the function $\ h(p)=h[\psi](p)=p \ln \psi(p) \ $, 
 (see, e.g., \cite{Buld Koz AMS} and references therein)
$$
h^*(v)= h^*[\psi](v) := \sup_{p \in \Dom[\psi]}(p v - p \ln \psi(p))
$$
where $ \ \Dom[\psi] \ $ denotes the domain of definition (and finiteness) for the function $ \ \psi(\cdot). \ $

\vspace{1mm}
\ We get
\begin{equation} \label{Young Fen}
T_{\upsilon}(t) \le \exp(-h^*(\ln t)), \ \ t \ge e.
\end{equation}
Note that the inverse conclusion, under suitable natural conditions, is true (see, e.g., \cite{Kozachenko-Ostrovsky 1985}). \par
 Notice also that these spaces coincide, up to equivalence of the norms and under appropriate conditions, with the so-called {\it exponential Orlicz} spaces (see, e.g.,
\cite{Kozachenko-Ostrovsky 1985,Kozachenko at all 2018,Ostrovsky1999,Ostrovsky HIAT}).

 \vspace{4mm}

\section{Main result.}

\vspace{4mm}

 \hspace{3mm} 
 Here we extend the estimate \eqref{important relation} from the ordinary Lebesgue-Riesz spaces to the Grand Lebesgue Spaces.
 Recall that here $ \ X = \mathbb R_+  \ $ equipped with the
 classical Lebesgue measure $ \ d \mu = dx. \ $ \par

\begin{theorem}
Let $ \ G\psi_j(a_j, b_j), \ j = 1,2,\ldots,m $, where $ \ 1 \le a_j < b_j \le \infty $, be Grand Lebesgue Spaces with generating functions $ \ \psi_j(\cdot)$ and put $\vec{\psi}=(\psi_1,\ldots,\psi_m)$.

\noindent Let $\beta_m[Q, \vec{\psi}](p)$ defined by
\begin{equation} \label{final gener fun}
\beta_m[Q, \vec{\psi}](p) \stackrel{def}{=} \inf \left\{\Theta_m(\vec{p}) \cdot \prod_{j = 1}^m \{ \ ||f_j||_{G\psi_j(a_j,b_j)} \cdot \psi_j(p_j) \ :  \ \sum_{j=1}^m \frac{1}{p_j} = \frac{1}{p} \ ; \ p_j \in (a_j,b_j)\right\},
\end{equation}
where $\Theta_m(\vec{p})$ is defined in \eqref{key value} and the kernel $Q$ is homogeneous of degree $-m$.\\[1mm]
\noindent Then, the multilinear integral operator defined in \eqref{multilin} satisfies
\begin{equation} \label{main estim}
||M[Q](\vec{f})||_{G\beta_m[Q, \vec{\psi}]} \le 1,
\end{equation}
with correspondent tail estimation (\ref{Young Fen}). Moreover, this estimation (\ref{main estim}) is exact still for the classical Lebesgue-Riesz spaces.

\end{theorem}

\begin{proof}
Let us consider the Grand Lebesgue Spaces
$ G\psi_j(a_j, b_j), \ j = 1,2,\ldots,m $, where $ \ 1 \le a_j < b_j \le \infty$, with generating functions $ \ \psi_j(\cdot);  \ $ we have

\begin{equation*} \label{key condit}
f_j \in G\psi_j(a_j, b_j)\hspace{3mm}\Leftrightarrow  \hspace{3mm} ||f_j||_{G\psi_j(a_j,b_j)}  < \infty.
\end{equation*}
We deduce, therefore, by virtue of the direct definition of the Grand Lebesgue norm
\begin{equation} \label{p inequality}
||f_j||_{p_j} \le  ||f_j||_{G\psi_j(a_j,b_j)} \cdot \psi_j(p_j), \ \ p_j \in ( a_j, b_j ), \ \ j = 1,2,\ldots,m.
\end{equation}
It follows immediately from  (\ref{important relation})
\begin{equation} \label{M ineq}
||M||_p \le \Theta_m(\vec{p}) \cdot \prod_{j = 1}^m \{ \ ||f_j||_{G\psi_j(a_j,b_j)} \cdot \psi_j(p_j) \ \}
\end{equation}
Finally, introduce the generating function defined in \eqref{final gener fun} and put 
$$ (r,s) := \{p: \beta_m[Q, \vec{\psi}](p) < \infty  \}.$$ 
The estimation \eqref{M ineq} yields \eqref{main estim}, with  correspondent tail estimation (\ref{Young Fen}). Moreover, (\ref{main estim}) is exact still for the classical Lebesgue-Riesz spaces.

\end{proof}

\vspace{6mm}

\emph{\textbf{\footnotesize Acknowledgements}.} {\footnotesize M.R. Formica is member of Gruppo
Nazionale per l'Analisi Matematica, la Probabilit\`{a} e le loro Applicazioni (GNAMPA) of the Istituto Nazionale di Alta Matematica (INdAM) and member of the UMI group \lq\lq Teoria dell'Approssimazione e Applicazioni (T.A.A.)\rq\rq. 
}

\vspace{6mm}

\end{document}